\documentclass[a4paper,11pt]{article}
\usepackage[english]{babel}
\usepackage[T1]{fontenc}
\usepackage[utf8]{inputenc}
\usepackage{amsmath}
\usepackage{amsthm}
\usepackage{amsfonts}
\usepackage{mathrsfs}
\usepackage{mathtools}
\usepackage{amssymb}
\usepackage{graphicx}
\usepackage{geometry}
\usepackage{textcomp}
\usepackage{tikz}
\usetikzlibrary{shapes}
\usepackage{stmaryrd}
\usepackage{wasysym}
\usepackage{titlesec}

\newcommand{\N}{\mathbf{N}}
\newcommand{\esp}{\mathbf E}
\newcommand{\calP}{\mathcal P}
\newcommand{\PP}{(\P_t)_{t\geq0}}
\newcommand{\XX}{(X_t^x)_{t\geq0}}
\newcommand{\calL}{\mathcal L}
\newcommand{\diag}{\mathrm{diag}}
\newcommand{\R}{\mathbf{R}}
\newcommand{\Q}{\mathbf{Q}}

\newcommand{\inv}[1]{\ensuremath{\frac{1}{#1}}}

\newcommand{\Inv}{^{-1}}

\newcommand{\ent}{\ensuremath{\mathrm{Ent}}}

\renewcommand{\L}{\mathbf L}
\renewcommand{\P}{\mathbf P}

{\theoremstyle{definition}
\newtheorem{defin}{Definition}[section]}
{\theoremstyle{definition}
}
{\theoremstyle{definition}
\newtheorem{req}[defin]{Remark}}
{\theoremstyle{definition}
\newtheorem{pdefin}[defin]{Proposition-definition}}
\newtheorem{prop}[defin]{Proposition}
\newtheorem{thm}[defin]{Theorem}

\newtheorem{lem}[defin]{Lemma}
\newtheorem*{thm0}{Theorem}

\titleformat{\section}[block]{\scshape}{\thesection.}{3pt}{\filcenter}{}

\title{A Feynman-Kac approach for Logarithmic Sobolev Inequalities}
\author{Clément Steiner}

\begin{document}
\maketitle
\section*{Abstract}

\begin{center}
\begin{minipage}{0.9\textwidth}
	This note presents a method based on Feynman-Kac semigroups for logarithmic Sobolev inequalities. It follows the recent work of Bonnefont and Joulin on intertwining relations for diffusion operators, formerly used for spectral gap inequalities, and related to perturbation techniques. In particular, it goes beyond the Bakry-Émery criterion and allows to investigate high-dimensional effects on the optimal logarithmic Sobolev constant.	
	The method is illustrated on particular examples (namely Subbotin distributions and double-well potentials), for which explicit dimension-free bounds on the latter constant are provided. We eventually discuss a brief comparison with the Holley-Stroock approach.
\end{minipage}
\end{center}

\section{Introduction}{\let\thefootnote\relax\footnote{\noindent \textbf{Key words}: Diffusion processes, Feynman-Kac semigroups, Logarithmic Sobolev inequalities, Perturbed Functional Inequalities.\newline
Mathematics Subject Classification (2010): 39B62, 47D08, 60J60 }}

Since their introduction by Gross in 1975, the Logarithmic Sobolev Inequalities ($LSI$) became a widely used tool in infinite dimensional analysis. Initially studied in relation to the hypercontractivity property for Markov semigroups, they turned out to be prominent in many various domains, at the interface of analysis, probability theory and geometry (one of the best example of such prominence being their use in Perelmann's proof of Poincaré's conjecture in \cite{Per02}).

For $\mu$ a probability measure on the Euclidean space $(\R^d,|\cdot|)$, this inequality provides a control on the entropy of any smooth function $f$ in term of its gradient: 
\[\ent_\mu(f^2)\leq c\int_{\R^d}|\nabla f|^2d\mu,\]
for some $c>0$, where $\ent_\mu(f^2)=\int_{\R^d}f^2\log(f^2)\;d\mu-\left(\int_{\R^d}f^2\;d\mu\right)\log\left(\int_{\R^d}f^2\;d\mu\right)$. The optimal constant for the latter inequality to hold, often called the \emph{logarithmic Sobolev constant} and denoted $c_{LSI}(\mu)$, is of primary importance in the study of the measure $\mu$, since it encodes many of its properties. For instance, among many results in this area, Otto and Villani established in \cite{OV00} a connection between $LSI$ and some transportation inequalities (see also the related work by Bobkov and Götze in \cite{BG99}), and Herbst provided a powerful argument that links $LSI$ to Gaussian concentration inequalities (see the lecture notes by Ledoux \cite{Led99} for more details and his reference monograph \cite{Led01} about concentration of measure).

The case where $\mu$ is the invariant measure of some Markov process is also of great interest. For example, apart from Gross' initial results on hypercontractivity in \cite{Gro75}, $c_{LSI}(\mu)$ encodes the decay in entropy of the related semigroup, and is linked to the Fisher information (defined for a positive function $f$ as $\int_{\R^d} |\nabla\sqrt{f}|^2d\mu$) through de Bruijn's identity. Significant advances in this setting were due to Bakry and Émery in \cite{BE85}, who stated their eponymous criterion, also known as "curvature-dimension criterion", that connects the logarithmic Sobolev inequality (and many functional inequalities) to geometric properties of $\mu$. We refer to \cite{BGL14} for a comprehensive overview of this theory.\smallskip

Although the value of $c_{LSI}(\mu)$ is key in the study of $\mu$, its exact value is hardly ever known explicitly. Bakry-Émery theory provides sharp estimates on this constant for some log-concave measures, assumption that might be weakened according to some perturbation arguments. More precisely, although the Bakry-Émery criterion is defined in a more general situation, it can be reformulated conveniently in the Euclidean setting as follows.

\begin{thm0}[Bakry-Émery, \cite{BE85}]
Assume that $\mu(dx)\propto e^{-V(x)}dx$, for some smooth potential $V$. If there exists some $\rho>0$ such that $\nabla^2V(x)\geq \rho I_d$ for any $x\in\R^d$ (the Hessian matrix of $V$ is uniformly bounded from below as a symmetric matrix), then $\mu$ satisfies a $LSI$ with constant $2/\rho$.
\end{thm0}

We refer to \cite{BGL14} \S 5.7 for the general curvature-dimension criterion. We shall stick from know on to the assumption that $\mu(dx)\propto e^{-V(x)}dx$, for some smooth potential $V$. In particular, this bound is sharp for the standard Gaussian distribution $\gamma$, providing $c_{LSI}(\gamma)=2$ (whatever the dimension of the underlying space is). Unfortunately, this criterion fails as soon as $V$ is not uniformly convex. Yet, if this "lack of convexity" can be balanced by a bounded transformation, one may use perturbation techniques, such as the well-known Holley-Stroock method.

\begin{thm0}[Holley-Stroock, \cite{HS87}]
Assume that $d\mu\propto e^\Phi d\nu$ where $\nu$ is a probability measure that satisfies a $LSI$ and $\Phi$ is continuous and bounded. Then $\mu$ satisfies a $LSI$ with $c_{LSI}(\mu)\leq e^{2(\sup(\Phi)-\inf(\Phi))}c_{LSI}(\nu)$.
\end{thm0} 

Note that perturbation by unbounded functions (under for example growth assumptions) has been studied, see for example \cite{BLW07}. However, authors in the latter explain that their method weakens the inequality as soon as the perturbation is not bounded. Nevertheless, the $LSI$ can be preserved by unbounded perturbation in some specific cases, as will be developed in this article.

 Apart frome perturbation, stability of $LSI$ by tensorization is also a key property of such inequalities, since it exhibits dimension-free behaviours for product measures, but fails in general to provide efficient bounds beyond this case. In particular, one may wish to keep track of the geometry of $\mu$ (dimension of the space, log-concavity, curvature, etc.) through $c_{LSI}(\mu)$, which can be difficult in many settings (as will be discussed in Section 4). For further reading, we refer to \cite{BGL14} \S 5.7 for detailed results and to the remarkably synthetic monograph \cite{ABC+00} for a broader introduction. \medskip

In this note, we provide a probabilistic approach based on the study of some Feynman-Kac semigroups to derive new estimates on the logarithmic Sobolev constant. It follows the recent work of Bonnefont and Joulin involving intertwinings and functional inequalities of spectral flavour \cite{BJ14,BJ19} and extends their approach to the latter. A somewhat similar approach can also be found in the recent work of Sturm and his collaborators on metric measure spaces \cite{BHS19}. Let us give an overview of our main results. They will be properly stated and proved in Section~3.

We first show a representation theorem for Feynman-Kac semigroups acting on gradient fields. Namely, for a perturbation function $a$ satisfying some regularity and growth assumptions, the following result holds.

\begin{thm0}
There exist a stochastic process $(X_{t,a})_{t\geq0}$, a martingale $(R_{t,a})_{t\geq0}$ and a matrix-valued process $(J_t^{X_a})_{t\geq0}$ such that for any smooth function $f$, one has
\[\calP_t^{\nabla^2V}(\nabla f)=\esp[R_{t,a}J_t^{X_a}\nabla f(X_{t,a})],\]
where $(\calP_t^{\nabla^2V})_{t\geq0}$ is the Feynman-Kac semigroup of interest.
\end{thm0}

The invariant measure of the above process is known and closely related to $\mu$. The martingale is given by Girsanov's theorem, while the matrix-valued process can be seen as the Jacobian matrix of $(X_{t,a})_{t\geq0}$ (with respect to the initial condition).

Note that this formula can be related to other forms of derivatives of heat semigroups, including for example the well-known Bismut formula (as presented for instance in \cite{EL94}). Originally derived using Malliavin calculus (see  \cite{Bis84}), Elworthy and Li emphasize in \cite{EL94} a more geometric approach and our proof relates to the differentiation of the flow of some stochastic differential equation (as presented for example in \cite{Pro04} \S V.7).  Yet, the above expression as a Feynman-Kac semigroup acting on a gradient field is particularly suitable when one aims to infer a logarithmic Sobolev inequality. In particular, this probabilistic representation allows to obtain Grönwall-type estimates on the semigroup, that lead to a new criterion for $LSI$. Namely, for a perturbation function $a$ satisfying some growth assumptions, we can define a curvature $\kappa_a\in\R$ (depending on $a$ and $\nabla^2V$) which provides a Bakry-Émery-like condition.
\begin{thm0}
If $\kappa_a>0$, then $\mu$ satisfies a $LSI$ with $c_{LSI}(\mu)\leq C_a/\kappa_a$, for some $C_a>0$.
\end{thm0}
This result indeed encompasses the Bakry-Émery criterion (taking $a\equiv1$). Note that we derive, in the specific case of monotonic functions, a very similar result, yet allowing the function $a$ to be unbounded. 

The choice of $a$ (provided that technical assumptions are satisfied) in the latter theorem is rather free, so that one expects to take it such as $C_a/\kappa_a$ is minimal (to get the sharpest bound on $c_{LSI}(\mu)$). The precise value of $C_a$ and $\kappa_a$ and their behaviour with respect to $a$ are discussed in more details around two examples.

The first one is the quadric potential, that is $V(x)=|x|^4/4$; the second is the double-well: $V(x)=|x|^4/4-\beta|x|^2/2$ (for $\beta\in(0,1/2)$). Bakry-Émery criterion fails in both cases, yet our main theorem applies and we manage to infer the following behaviour of $c_{LSI}(\mu)$.

\begin{thm0}
\begin{itemize}
\item Assume that $V(x)=|x|^4/4$. Then $\mu$ satisfies a $LSI$ and $c_{LSI}(\mu)$ does not depends on the dimension.
\item Assume that $V(x)=|x|^4/4-\beta|x|^2/2$, $\beta\in(0,1/2)$. Then $\mu$ satisfies a $LSI$ and $c_{LSI}(\mu)$ only depends on $\beta$.
\end{itemize}
\end{thm0}

We briefly compare both results with the Holley-Stroock method, and provide explicit constants.\medskip

The article is organised as follows. We introduce in Section~2 the framework of the paper, along with some results about intertwinings and Feynman-Kac semigroups. In Section~3, we properly state and prove our main results and discuss a comparison with the Holley-Stroock approach. Finally, Section~4 is devoted to examples, where explicit constants and detailed computations are provided.

\section{Basic framework}

In this first section, we recall the framework of our analysis, basic results and definitions about intertwinings and Feynman-Kac semigroups (as introduced in \cite{BJ14,ABJ18}).

\subsection{Setting}

The whole analysis shall be performed on the $d$-dimensional Euclidean space $(\R^d,|\cdot|)$, for $d\in\N^\star$. We let $\mathcal C^\infty(\R^d,\R)$ and $\mathcal C^\infty(\R^d,\R^d)$ be respectively the set of infinitely differentiable functions and vector fields on $\R^d$, and let $\mathcal C^\infty_c(\R^d,\R)$ and $\mathcal C^\infty_+(\R^d,\R)$ denote respectively the set of compactly supported and positive $\mathcal C^\infty$ functions on $\R^d$. We endow those spaces with the supremum norm $\|\cdot\|_\infty$. We consider throughout this article a probability measure $\mu$ on $\R^d$ whose density with respect to the Lebesgue measure is proportional to $e^{-V}$, for some potential $V$ at least twice differentiable. To this measure, one can associate a Markov diffusion operator defined as
\[ \L=\Delta-\nabla V\cdot\nabla, \]
where $\Delta$ and $\nabla$ respectively stand for the usual Laplace operator and gradient on $\R^d$. The flow of the equation $\partial_tu=\L u$ over $\R_+$ defines a Markov semigroup $\PP$, invariant with respect to $\mu$, which is, under standard assumptions on $V$, ergodic in $L^2(\mu)$. Such assumptions include for example that $\L$ vanishes only for constant functions and the latter are stable by $\PP$. See \cite{BGL14} \S 3.1.9 for a general result. Moreover, this semigroup describes the dynamics of a diffusion process $\XX$ that solves the following Stochastic Differential Equation (SDE):
\begin{equation}\label{SDE}
dX_t^x=\sqrt2\,dB_t-\nabla V(X_t^x)dt,\quad X_0^x=x\in\R^d\ \mathrm{a.s.},\tag{$E$}
\end{equation}
where $(B_t)_{t\geq0}$ denotes the standard $d$-dimensional Brownian motion. All stochastic processes are defined on some probability space $(\Omega,\mathcal F,\P)$, and we let $(\mathcal F_t)_{t\geq0}$ denote the natural (completed) filtration associated to $(B_t)_{t\geq0}$. Under mild assumptions on $V$, this process is non-explosive and converges in distribution towards $\mu$, its invariant distribution. Moreover, regularity of $V$ ensures that $x\mapsto X_t^x$ is (at least) differentiable over $\R^d$, for any $t\geq0$. See Remarks~\ref{req_explo} and \ref{req_reg} at the end of this section for more informations and references about non-explosion and regularity w.r.t. the initial condition.

In addition, $\L$ is symmetric on $\mathcal C^\infty_c(\R^d,\R)$ with respect to $\mu$, and the integration by parts formula rewrites as follows: for any $f,g\in\mathcal C^\infty_c(\R^d,\R)$,
\[\int_{\R^d}f\L g\,d\mu=\int_{\R^d}g\L f\,d\mu=-\int_{\R^d}\nabla f\cdot\nabla g\,d\mu.\]
In particular, $\L$ is non-positive on $\mathcal C^\infty_c(\R^d,\R)$. Hence by completeness, this operator admits a unique self-adjoint extension (which shall still be denoted $\L$) on some domain $\mathcal D(\L)\subset L^2(\mu)$ for which  $\mathcal C^\infty_c(\R^d,\R)$ is a core, \textit{i.e.} is dense for the norm induced by $\L$ (see \cite{BGL14} \S 3.1.8 and thereafter for more precise informations).\smallskip

Finally, let us recall the definition of the logarithmic Sobolev inequality we will refer to.

\begin{defin}
The measure $\mu$ is said to satisfy a \emph{Logarithmic Sobolev Inequality} (in short $LSI$) with constant $c>0$ if for any $f\in\mathcal C_c^\infty(\R^d,\R)$ one has
\[\ent_\mu(f^2)\leq c\int_{\R^d}|\nabla f|^2d\mu.\]
We let $c_{LSI}(\mu)$ denote the optimal constant in the latter inequality, which we may as well refer as the \emph{logarithmic Sobolev constant}.
\end{defin}

\begin{req}[Embeddings and integrability]
The integration by parts formula entails that the bilinear form $(f,g)\mapsto-\int_{\R^d} f\L g\,d\mu$ extends likewise on some domain in which $\mathcal D(\L)$ is dense for the $H^1(\mu)$ norm (see \cite{BGL14} \S 3.3.2 for more detailed statements). Hence, in some way, the $LSI$ may be seen as a continuous embedding of $H^1(\mu)$ into some Orlicz space (see \cite{RZ06}), in quite a similar way as Sobolev inequalities provide a continuous (and compact, with Rellich-Kondrachov theorem) embedding of $H^1(\mu)$ into some $L^p(\mu)$ spaces (see for example \cite{BGL14} \S 6.4). Similarly, through Herbst's argument, $LSI$ implies that the square of 1-Lipschitz functions is exponentially integrable (see \cite{BGL14} \S 5.4), and thus Gaussian concentration for $\mu$, whereas a Sobolev inequality implies that such functions are actually bounded in $H^1(\mu)$ (see \cite{BGL14} \S 6.6). One may as well compare both Sobolev and logarithmic Sobolev inequalities to the (weaker) Poincaré inequality, we refer the interested reader to \cite{BGL14} \S 4.4 for further information.

\end{req} 

We end this setting section with some details and references about diffusion processes.

\begin{req}[Diffusion processes: non-explosion]\label{req_explo}
The explosion time of the process $(X_t)_{t\geq0}$ is defined as $\tau_e=\inf\{t\geq0:\limsup_{s\to t}|X_s|=+\infty\}$ (the definition is quite similar to the classical ODE one, except that $\tau_e$ is here a stopping time w.r.t. $(\mathcal F_t)_{t\geq0}$). The process is said to be non-explosive (in finite time) as soon as $\tau_e$ is almost surely infinite. This is actually equivalent to the \emph{mass preservation} for $\PP$, that is, $\P_t\mathbf 1=\mathbf 1$ for any $t\geq0$, with $\mathbf 1$ the constant function equal to 1 (understood as the increasing limit of a sequence of compactly supported $\mathcal C^\infty$ functions). Indeed, for any $t\geq 0$, $\P_t\mathbf 1=\P(t\leq \tau_e)$, so that $\P_t\mathbf 1=\mathbf 1$ for any $t\geq0$ if and only if $\tau_e=+\infty$ almost surely. This property is somewhat easier to handle, and Bakry inferred in \cite{Bak86} the following criterion: $(X_t)_{t\geq0}$ is non-explosive as soon as there is $\rho\in\R$ such that $\nabla^2V(x)\geq \rho I_d$ for any $x\in\R^d$. Note that $\rho$ is not required to be positive, making it a very general condition. Roughly speaking, it states that $V$ should not be "too concave".

From now on, we will assume that the latter is satisfied. This is for example true for $V(x)\propto|x|^\alpha$ for $\alpha >1$, as shall be made clear in Section 4.
\end{req}
\begin{req}[Diffusion processes: initial condition]\label{req_reg}
As quickly mentioned above, as long as the process $(X_t)_{t\geq0}$ is non-explosive and $\nabla V$ is smooth enough, the function $x\mapsto X_t^x(\omega)$ (for any fixed $t\geq0$ and almost any fixed $\omega\in\Omega$) is differentiable on $\R^d$. Actually, as mentioned in \cite{BGL14} \S B.4, this application is as smooth as $\nabla V$ is (up to the explosion time). In the following, we may only focus on the first order derivative (also know as tangent process or tangent flow), but general results for any order of differentiation can be found in \cite{Pro04} \S V.7 (Theorems 39 and 40).
\end{req}

\subsection{Intertwinings} We now focus on intertwinings (for a comprehensive introduction, see \cite{BJ14,ABJ18}). Basically, we are interested in commutation relations between gradients and Markov generators, which give rise to the so-called Feynman-Kac semigroups. In the following proposition, we introduce some notation related to tensor operators and recall a chain rule commutation formula.

\begin{pdefin}\label{pdefin_simple_inter}
In the following, we let $\calL$ denote the tensorized operator $\L^{\otimes d}$ and $(\calP_t)_{t\geq0}$ be the associated Markov semigroup, that both act on vector fields. For $F=(F_1,\dots,F_d)\in\mathcal C^\infty(\R^d,\R^d)$, they write as
\[\calL F=(\L F_1,\dots,\L F_d)\ \text{ and }\ \calP_tF=(\P_t F_1,\dots,\P_t F_d).\]
For $f\in\mathcal C^\infty(\R^d,\R)$, we recall the intertwining relation:
\[\nabla\L f=(\calL-{\nabla^2V})(\nabla f),\]
where $\nabla^2V\cdot\nabla f$ is the standard matrix-vector product. Similarly, the  Feynman-Kac semigroup $(\calP_t^{\nabla^2V})_{t\geq0}$ associated to $\calL-{\nabla^2V}$ satisfies the following identity:
\[\nabla\P_tf=\calP_t^{\nabla^2V}(\nabla f),\quad t\geq0,\]
provided that $f$ has compact support.
\end{pdefin}

This idea takes roots in various works in differential geometry and operators analysis, and relates (in some more general setting) to the Bochner-Lichnerowicz-Weitzenb\"ock formula (see \cite{Bou90} for an enthusiast introduction). See also the works around Witten Laplacians arising in statistical mecanics, for which we refer to Helffer's monograph \cite{Hel02}.

\begin{req}
We can still define the Feynman-Kac semigroup associated to $\calL$ and a general smooth map $M:\R^d\to\mathcal M_d(\R)$ as the flow of the following PDE system:
\[\left\{\begin{array}{rcl}
	\partial_tu & = & (\calL-M)u\\
	u(0,\cdot)&=&u_0,
	\end{array}
	\right. \]
denoted by $(\calP_t^M)_{t\geq0}$, provided that solutions to this system exist at any time. Such an extension will be implicitly used later.
\end{req}

\begin{req}
The original Feynman-Kac formula, that arises in quantum mechanics, is stated for scalar-valued functions $f$ and $m$ and writes as follows (see \cite{BGL14} \S 1.15.6):
\[\calP_t^{m}f=\esp\left[f(X_t)e^{-\int_0^tm(X_s)ds}\right].\]
We call $(\calP_t^{\nabla^2V})_{t\geq0}$ a Feynman-Kac semigroup by analogy with this case (which shall clearly appear in the following), yet the representation of $(\calP_t^{\nabla^2V})_{t\geq0}$ does not write as simply as the above. This is the object of the next section.
\end{req}

\section{Main results}

In this section, we state and prove our main results in two steps: we first provide a representation theorem, related to Feynman-Kac semigroups, then apply it to estimates on the logarithmic Sobolev constant.

In \cite{Wang14}, Wang developed a somehow similar approach in the framework of manifolds with boundaries, based on the Girsanov's theorem for reflected processes. Yet in our case, we take advantages of some properties of the semigroup, namely invariance and ergodicity.

\subsection{Representation of Feynman--Kac semigroups}

This first part is devoted to the main representation theorem we shall make use of. It is presented for Feynman-Kac semigroups acting on gradients, but still holds for more general vector fields (in which case the proof relies on a classical martingale argument). 

The perturbation technique that will be set up in the next section strongly relies on a Girsanov representation of the semigroup $\PP$. To this end, we introduce a smooth perturbation function in $V$ and study the relation between $(X_t)_{t\geq0}$ and the process obtained from this new potential.

\begin{defin}\label{def_weight_proc}
	Let $a\in\mathcal C_+^\infty(\R^d,\R)$. We let $(X_{t,a})_{t\geq0}$ denote the solution of the SDE
	\[dX_{t,a}=\sqrt 2dB_t-\nabla V_a(X_{t,a})\,dt,\]
	where $V_a=V+\log(a^2)$.
\end{defin}
Straightforward computations show that the generator of this process writes down 
\[\L_a=\L-2\frac{\nabla a}{a}\cdot\nabla,\]
and we let $(\P_{t,a})_{t\geq0}$ denote the associated Markov semigroup (in particular, for any $f\in\mathcal C_c^\infty(\R^d,\R)$, $\P_{t,a}f=\esp[f(X_{t,a})]$). Moreover, if $\mu_a$ is defined such that $d\mu_a/d\mu=1/a^2$, then $(\P_{t,a})_{t\geq0}$ is $\mu_a$-invariant and $\L_a$ is (essentially) self-adjoint in $L^2(\mu_a)$ (see as well \cite{BGL14} \S 3.1.8). Note that $\mu_a$ is not a probability (or even finite) measure a priori.

Provided that everything is well-defined, the intertwining relation of Proposition-definition \ref{pdefin_simple_inter} for semigroups is also available for $\P_{t,a}f$, and writes as follows:
\[\nabla\P_{t,a}f=\calP_{t,a}^{\nabla^2V_a}(\nabla f).\]

Before we state the main theorem of this section, let us define a condition on the perturbation function that naturally arises in the computations involving Girsanov's theorem.

\begin{defin}
A function $a\in\mathcal C_+^\infty(\R^d,\R)$ is said to satisfy the \emph{(G) condition} whenever $|\nabla a|/a$ is bounded.
\end{defin}

We can now turn to the representation result (the first theorem stated in the introduction). To fix the ideas, we may write down the initial condition in the following statements, and omit it in the proofs.

\begin{thm}\label{cor_WFK}
	Let $f\in\mathcal C_c^\infty(\R^d,\R)$ and $a\in\mathcal C^\infty_+(\R^d,\R)$ satisfying \textit{(G)}. Then for any $t\geq0$, $x\in\R^d$,
	\[\calP_t^{\nabla^2V}(\nabla f)(x)=\esp[R_{t,a}^xJ_t^{X^x_a}\nabla f(X^x_{t,a})], \]
	where $(R_{t,a}^x)_{t\geq0}$ is a martingale with respect to $(\mathcal F_t)_{t\geq0}$ defined as
	\[R_{t,a}^x=\frac{a(X^x_{t,a})}{a(x)}\exp\left(-\int_0^t\frac{\L_a a(x)}{a(x)}(X^x_{s,a})\;ds\right),\quad t\geq0,\]
	and $(J_t^{X^x_a})_{t\geq0}$ is a matrix-valued process that solves
	\[\left\{
	\begin{array}{rl}
	dJ_t^{X^x_a}&=-J_t^{X^x_a}\nabla^2V(X^x_{t,a})dt,\quad t>0\\
	J_0^{X^x_a}&=I_d.
	\end{array}\right.\]
\end{thm}

As mentioned before, this result is based on Girsanov's theorem (see \cite{Pro04} \S III.8~Theorem~46 for a proper statement). Hence, before we turn to its proof, we need the following lemma, that establishes a relation between the Markov semigroups $\PP$ and $(\P_{t,a})_{t\geq0}$.

\begin{lem}\label{lem_Girs}
	Let  $a\in\mathcal C_+^\infty(\R^d,\R)$ satisfying the $(G)$ condition. Then for any function $f\in\mathcal C_c^\infty(\R^d,\R)$, any $t\geq0$, $x\in\R^d$:
	\[\P_tf(x)=\esp[f(X^x_t)]=\esp\left[R^x_{t,a}f(X^x_{t,a})\right],\]
	where $(R_{t,a}^x)_{t\geq0}$ is the $(\mathcal F_t)_{t\geq0}$-martingale defined above.
\end{lem}
\begin{proof}
	We first set up a suitable exponential martingale before we identify the involved probability distributions with Girsanov's theorem. For the sake of legibility, the initial condition shall be omitted in the following.\medskip
	
	We first apply It\={o}'s formula to $\log(a(X_{t,a}))$:
	\[\log a(X_{t,a})=\log a+\sqrt2\int_0^t\nabla(\log a(X_{s,a}))\cdot dB_s+\int_0^t\L_a(\log a)(X_{s,a})\;ds.\]
	Note that this decomposition is fairly general and is related to the martingale problem of Stroock and Varadhan, see \cite{SV06} Chap. 6.
	
	Expanding the right-hand side and taking exponential lead to the following expression for $R_{t,a}$: 
	\[R_{t,a}=\exp\left(\sqrt2\int_0^t\frac{\nabla a}{a}(X_{s,a})\cdot dB_s-\int_0^t\left|\frac{\nabla a}{a}\right|^2(X_{s,a})\;ds\right).\]
	The $(G)$ condition ensures through standard arguments that the right-hand side is a true $(\mathcal F_t)_{t\geq0}$-martingale, thus so is $(R_{t,a})_{t\geq0}$. From now on, we set $Y_{t,a}=\sqrt2\frac{\nabla a}{a}(X_{t,a})$.\medskip
	
	We let $\Q_a$ be the probability measure defined as
	\[\left.\frac{d\Q_a}{d\P}\right|_{\mathcal F_t}=R_{t,a}.\]
	According to Girsanov's theorem, the process $(\tilde{B}_t)_{t\geq 0}$ defined as
	\[\tilde B_t=B_t-\int_0^tY_{s,a}\;ds,\]
	is a $\Q_a$-Brownian motion. Furthermore, the process $(X_{t,a})_{t\geq0}$ solves the SDE
	\[dX_{t,a}=\sqrt2d\tilde B_t-\nabla V(X_{t,a})\;dt,\]
	hence the law of $X_{t,a}$ under $\Q_a$ coincides with the one of $X_t$ under $\P$. In particular, for any $f\in\mathcal C_c^\infty(\R^d,\R)$,
	\[\P_tf=\esp[f(X_t)]=\esp[R_{t,a}f(X_{t,a})],\]
	and the proof is complete.
\end{proof}

We can now prove Theorem \ref{cor_WFK}.

\begin{proof}
	Recall that under the aforementioned non-explosion assumptions, the diffusion process defined by Equation \eqref{SDE} is differentiable with respect to its initial condition (see Remarks~\ref{req_explo} and \ref{req_reg}), so that for any $t\geq0$:
	\begin{align*}
	\calP_t^{\nabla^2V}(\nabla f)&=\nabla\P_tf\\
	&=\esp[\nabla(f(X_t))]\\
	&=\esp[J_t^X\nabla f(X_t)],
	\end{align*}
	where $(J_t^X)_{t\geq0}$ denotes the (matrix-valued) tangent process to $(X_t)_{t\geq0}$ (that is, the Jacobian matrix of $X_t$ with respect to the initial condition). Differentiating with respect to the initial condition in the SDE \eqref{SDE} provides the following formula for $J_t^X$:
	\[J_t^X=I_d-\int_0^tJ_s^X\nabla^2V(X_s)\;ds.\]
	One can replace $X_s$ by $X_{s,a}$ in the previous expression, to define as well
	\[J_t^{X_a}=I_d-\int_0^tJ_s^{X_a}\nabla^2V(X_{s,a})\;ds.\]
	Note that the potential $V$ is unchanged in the equation. Lemma \ref{lem_Girs} implies then, since $R_{t,a}$ is scalar-valued,
	\[\esp[J_t^X\nabla f(X_t)]=\esp[R_{t,a}J_t^{X_a}\nabla f(X_{t,a})],\]
	and the proof is complete.
\end{proof}
\begin{req}
Note that the statement of Lemma~\ref{lem_Girs} generalises to functionals of the trajectory $X_{[0,t]}:=(X_s,0\leq s\leq t)$. More precisely, for $t>0$ and $F:\mathcal C^0([0,t],\R)\to\R$ a continuous map, Girsanov's theorem yields as well:
\[\esp[F(X_{[0,t]})]=\esp[R_{t,a}F(X_{[0,t],a})].\]
This extension is implicitly used in the previous proof, taking $F(X_{[0,t]})=J_t^X\nabla f(X_t)$.
\end{req}

\begin{req}
In dimension $d=1$, since gradients and functions are both 1-dimensional objects, Theorem \ref{cor_WFK} rewrites in a more standard way:
\[(\P_tf)'=\P_t^{V''}(f')=\esp\left[R_{t,a}f'(X_{t,a})\exp\left(-\int_0^tV''(X_{s,a})\,ds\right)\right].\]
This writing shall be useful when dealing with monotonic functions in dimension 1, as briefly discussed at the end of the next section.
\end{req}

\subsection{Logarithmic Sobolev inequalities}

In this section, we provide a Feynman-Kac-based proof of the logarithmic Sobolev inequality, stated for a scalar perturbation. The method can easily be refined to improve the bound on $c_{LSI}(\mu)$, for example when finer spectral estimates on the generator are available or for a restricted set of test functions. For instance, we adapt the proof to derive estimates in restriction to monotonic (positive) functions.

\subsubsection{General case}

\textit{Notation.} The proof of the following theorem requires some matrix analysis. Henceforward, if $A$ is a symmetric matrix, we let $\rho_-(A)$ denote its smallest eigenvalue. We may also use $M^T$ and $u^T$ to denote the usual transpose of a matrix $M$ or a vector $u$.\smallskip

The following result states the Bakry-Émery-like criterion mentioned in the introduction.

\begin{thm}\label{thm_LSI}
	Let $a\in\mathcal C_+^\infty(\R^d,\R)$. Define 
	\[\kappa_a=\inf\limits_{x\in\R^d}\left\{2\rho_-(\nabla^2V(x))-a\L(a\Inv)(x)\right\}.\]
	If $a$, $a\Inv$ and $|\nabla a|$ are bounded and $\kappa_a>0$, then for any $f\in\mathcal C_c^\infty(\R^d,\R)$,
	\[\ent_\mu(f^2)\leq\frac{4\|a\|_\infty\|a^{-1}\|_\infty}{\kappa_a}\int_{\R^d}|\nabla f|^2\;d\mu.\]
\end{thm}
\begin{proof}
	Let $f\in\mathcal C_c^\infty(\R^d,\R)$ be a non-negative function. Ergodicity and $\mu$-invariance of $\PP$ give:
	\[\ent_\mu(f)=-\int_{\R^d}\int_0^{+\infty}\partial_t\left(\P_tf\log\P_tf\right)\,dt\,d\mu=-\int_{\R^d}\int_0^{+\infty}\L[\P_tf]\log\P_tf\,dt\,d\mu.\]
	The integration by parts formula and the intertwining relation yield then:
	\[\ent_\mu(f)=\int_{\R^d}\int_0^{+\infty}\frac{|\nabla\P_tf|^2}{\P_tf}\,d\mu\,dt=\int_{\R^d}\int_0^{+\infty}\frac{\left|\calP_t^{\nabla^2V}(\nabla f)\right|^2}{\P_tf}\;dt\;d\mu.\]
	We focus on the numerator of the right-hand side. More precisely, we aim to cancel out $\P_tf$ at the denominator, which is made possible by Girsanov's theorem. Indeed, the assumptions on $a$ ensure that it satisfies the $(G)$ condition, and Theorem \ref{cor_WFK} leads to
	\[\calP_t^{\nabla^2V}(\nabla f)=\esp[R_{t,a}J_t^{X_a}\nabla f(X_{t,a})],\]
	which rewrites
	\[\calP_t^{\nabla^2V}(\nabla f)=2\esp\left[R_{t,a}^{1/2}J_t^{X_a}\nabla \sqrt f(X_{t,a})R_{t,a}^{1/2}\sqrt f(X_{t,a})\right].\]
	Cauchy-Schwarz' inequality with Lemma \ref{lem_Girs} finally entail
	\begin{align*}
	\left|\calP_t^{\nabla^2V}(\nabla f)\right|^2&\leq 4\esp\left[\left|R_{t,a}^{1/2}J_t^{X_a}\nabla \sqrt f(X_{t,a})\right|^2\right]\esp\left[R_{t,a}f(X_{t,a})\right]\\
	&=4\esp\left[\nabla\sqrt f(X_{t,a})^TJ_t^{X_a}R_{t,a}(J_t^{X_a})^T\nabla\sqrt f(X_{t,a})\right]\P_tf.
	\end{align*}
	This implies then for the entropy:
	\[\ent_\mu(f)\leq4\int_{\R^d}\int_0^{+\infty}\esp\left[\nabla\sqrt f(X_{t,a})^TJ_t^{X_a}R_{t,a}(J_t^{X_a})^T\nabla\sqrt f(X_{t,a})\right]\;dt\;d\mu.\]
	In order to recover the energy term in the $LSI$, one should provide some spectral estimates for $J_t^{X_a}R_{t,a}(J_t^{X_a})^T$. Define then
	\[J_t^a=J_t^{X_a}\exp\left(-\inv2\int_0^t\frac{\L_aa}{a}(X_{s,a})\,ds\right),\]
	which solves the following equation:
	\[dJ_t^a=-J_t^a\left(\nabla^2V(X_{t,a})-\inv2a\L(a\Inv)(X_{t,a})I_d\right)dt.\]
	Indeed, we have on the one hand:
	\[dJ_t^{X_a}=-J_t^{X_a}\nabla^2V(X_{t,a})\,dt,\]
	and on the other hand:
	\[d\left[\exp\left(-\inv2\int_0^t\frac{\L_aa}{a}(X_{s,a})\,ds\right)\right]=-\inv2\frac{\L_aa}{a}(X_{t,a})\exp\left(-\inv2\int_0^t\frac{\L_aa}{a}(X_{s,a})\,ds\right)\,dt.\]
	Moreover, $\L_a(a)/a=-a\L(a\Inv)$, so that both previous points and a chain rule give the expected formula. Since $J_t^{X_a}R_{t,a}(J_t^{X_a})^T=\frac{a(X_{t,a})}{a(x)}J_t^a(J_t^a)^T$, one should focus on spectral estimates for the latter term.\medskip
	
	Therefore, if we let $\varphi(t)=y^TJ_t^a(J_t^a)^Ty$, for some $y\in\R^d$, symmetry of $\nabla^2V$ entails
	\begin{align*}
	d\varphi(t)&=y^TdJ_t^a(J_t^a)^Ty+y^TJ_t^a(dJ_t^a)^Ty\\
	&=-y^TJ_t^a\left(\nabla^2V(X_{t,a})-\inv2a\L(a\Inv)(X_{t,a})I_d\right)(J_t^a)^Ty\,dt\\
	&\qquad-y^TJ_t^a\left(\nabla^2V(X_{t,a})-\inv2a\L(a\Inv)(X_{t,a})I_d\right)^T(J_t^a)^Ty\,dt\\
	&=-y^TJ_t^a\left(2\nabla^2V(X_{t,a})-a\L(a\Inv)(X_{t,a})I_d\right)(J_t^a)^Ty\,dt\\
	&\leq-\kappa_ay^TJ_t^a(J_t^a)^Ty\,dt=-\kappa_a\varphi(t)\,dt,
	\end{align*}
	by definition of $\kappa_a$. Hence, $\varphi(t)\leq e^{-\kappa_at}\varphi(0)$ for any $t\geq0$, which yields
	\[y^TJ_t^a(J_t^a)^Ty\leq e^{-\kappa_at}|y|^2.\]
	We can apply the previous inequality to $y=\sqrt{\frac{a(X_{t,a})}{a(x)}}\nabla\sqrt f(X_{t,a})$ to get
	\[\ent_\mu(f)\leq4\int_{\R^d}\int_0^{+\infty}e^{-\kappa_at}\esp\left[\frac{a(X_{t,a})}{a(x)}|\nabla\sqrt f(X_{t,a})|^2\right]\;dt\;d\mu,\]
	which rewrites
	\[\ent_\mu(f)\leq4\int_0^{+\infty}e^{-\kappa_at}\int_{\R^d}\inv a\P_{t,a}\left(a|\nabla\sqrt f|^2\right)\;d\mu\;dt.\]
	Recall that $d\mu_a/d\mu=1/a^2$. Then, since $a$ is bounded,
	\[\ent_\mu(f)\leq4\|a\|_\infty\int_0^{+\infty}e^{-\kappa_at}\int_{\R^d}\P_{t,a}\left(a|\nabla\sqrt f|^2\right)\;d\mu_a\;dt.\]
	One can use invariance of $\P_{t,a}$ with respect to $\mu_a$, then assumption on $\kappa_a$ to get
	\[\ent_\mu(f)\leq\frac{4\|a\|_\infty}{\kappa_a}\int_{\R^d}a|\nabla\sqrt f|^2\;d\mu_a.\]
	Finally, boundedness of $a\Inv$ entails
	\[\ent_\mu(f)\leq\frac{4\|a\|_\infty\|a\Inv\|_\infty}{\kappa_a}\int_{\R^d}|\nabla\sqrt f|^2\;d\mu,\]
	and the proof is complete replacing $f$ by $f^2$.
\end{proof}

\begin{req}
In terms of perturbation matrices (as presented in \cite{ABJ18} through weighted intertwinings) one has here $A=aI_d$. To take into account the geometry of $\nabla^2V$, a natural extension to this result would be to consider non-homothetic perturbations, for instance of the form $A=\diag(a_1,\dots,a_d)$, where $a_1,\dots,a_d\in\mathcal C^\infty_+(\R^d,\R)$ are distinct functions. In spite of many attempts, the above proof does not transpose to this case, and more general spectral estimates are besides much harder to derive. Generalisation of the representation result and Grönwall-like estimates for such perturbations would then allow an interesting extension to this result.
\end{req}

\begin{req}[Holley-Stroock criterion]\label{req_HS}
One may wish to compare this technique to the well-known Holley-Stroock method (introduced in \cite{HS87} for the Ising model). As a reminder, if $\nu$ is a probability measure that satisfies a $LSI$ and there exists $\Phi:\R^d\to\R$ a bounded continuous function such that $d\mu\propto e^\Phi d\nu$, then $\mu$ satisfies a $LSI$ and
\[c_{LSI}(\mu)\leq e^{2osc(\Phi)}c_{LSI}(\nu),\]
where $osc(\Phi)=\sup(\Phi)-\inf(\Phi)$. Note that $osc(\Phi)$ can poorly depend on the dimension, for example if $\Phi(x)=\sum_{i=1}^d\varphi(x_i)$, in which case $osc(\Phi)=d\cdot osc(\varphi)$. To stick to our framework, one might choose $\Phi=\log(a^2)$ for some bounded perturbation function $a\in\mathcal C_+^\infty(\R^d,\R)$. The above inequality becomes
\[c_{LSI}(\mu)\leq\|a\|_\infty^4\|a\Inv\|_\infty^4c_{LSI}(\mu_a),\]
so that Holley-Stroock method leads to show that $\mu_a$ satisfies a $LSI$. This is conveniently ensured as soon as $\mu_a$ satisfies the Bakry-Émery criterion, namely
\[\inf_{x\in\R^d}\{\rho_-(\nabla^2V_a(x))\}>0.\]
In terms of $a$ and $V$, the above condition rewrites explicitly:
\[\inf_{\R^d}\left\{\rho_-\left(\nabla^2V+\frac{2}{a}\nabla^2a-\frac{2}{a^2}\nabla a(\nabla a)^T\right)\right\}>0,\]
which shall be compared to the spectral estimates involved in $\kappa_a$, that can be expressed as:
\[\inf_{\R^d}\left\{\rho_-(\nabla^2V)+\frac{\Delta a}{a}-\nabla V\cdot\nabla a-\frac{2}{a^2}|\nabla a|^2 \right\}>0.\]
Both expressions do not compare to each other, yet the second one seems to be far more tractable, as it could be illustrated on various examples.
\end{req}

As mentioned before, the above proof can be adapted in some particular cases to improve the estimate on $c_{LSI}(\mu)$. In the following, we thus study the restriction of the latter to monotonic (positive) functions.

\subsubsection{Monotonic functions}

\begin{defin}
	A measurable function $f:\R^d\to\R$ is said to be \emph{monotonic} (in each direction) if for any $i=1,\dots,d$, for any fixed $(x_1,\dots,x_{i-1},x_{i+1},\dots,x_d)\in\R^{d-1}$  , $f_i:x_i\mapsto f(x_1,\dots,x_d)$ is monotonic.
	
	In particular, if $f$ is differentiable, then $f$ is monotonic if and only if $\partial_if$  has a constant sign on $\R^d$ for any $i$. 
\end{defin}

\begin{req}
	In the following, we shall focus on smooth functions $f$ such that all $f_i$ are non-decreasing (resp. non-increasing). In such cases, $f$ will be said to be itself non-decreasing (resp. non-increasing).
\end{req}
\begin{defin}[\textit{(BM)} condition]
	Given the potential $V$, a function $a\in\mathcal C^\infty_+(\R^d,\R)$ is said to satisfy the \emph{Bakry-Michel condition} (in short \textit{(BM)}) if:
	\begin{enumerate}
		\item for any $i,j\in\llbracket 1,d\rrbracket$, $i\neq j$, $\partial^2_{ij}V_a\leq0$;
		\item for any $i\in\llbracket 1,d\rrbracket$, $\sum_{j=1}^{d}\partial^2_{ij}V_a$ is upper bounded,
	\end{enumerate}
\end{defin}
The following proposition is one of the main arguments that allows to improve the estimate on $c_{LSI}(\mu)$ for monotonic functions.
\begin{prop}\label{prop_monotone}
	Let $f\in\mathcal C_+^\infty(\R^d,\R)$ and $a\in\mathcal C_+^\infty(\R^d,\R)$ satisfying (BM). Assume furthermore that $f$ and $a$ are both non-decreasing. Then
	\[ \P_{t,a}f\leq\P_tf,\quad t\geq0. \]
\end{prop}
This proposition is based on a lemma provided by Bakry and Michel in \cite{BM92}, used initially to infer some FKG inequalities in $\R^d$.
\begin{lem}\label{lem_BM}
	Let $M:\R^d\to\mathcal M_d(\R)$ be a measurable map such that $M_{ij}\leq0$ for any $i\neq j$ and $\sum_{j=1}^{d}M_{ij}$ is upper bounded for any $i$, and let $F$ be a smooth vector field on $\R^d$. Then all components of $\calP_t^MF$ are non-negative whenever all components of $F$ are so.
\end{lem}
We refer the reader to \cite{BM92} for the proof. We can now provide a proof of Proposition~\ref{prop_monotone}.

\begin{proof}
	The proof relies on very classical techniques. Let $t\geq0$ be fixed and take $f\in\mathcal C_+^\infty(\R^d,\R)$ a non-decreasing function. Define, for any $s\in[0,t]$,
	\[ \Psi(s)=\P_s(\P_{t-s,a}f). \]
	Since $\Psi(0)=\P_{t,a}f$ and $\Psi(t)=\P_tf$, we aim to prove that $\Psi$ is non-decreasing. One has, for any $s\in[0,t]$,
	\[ \Psi'(s)=\P_s[(\L-\L_a)\P_{t-s,a}f], \]
	which rewrites accordingly
	\[ \Psi'(s)=\P_s\left[\frac{\nabla a}{a}\cdot\nabla\P_{t-s,a}f\right]=\P_s\left[\frac{\nabla a}{a}\cdot\calP_{t-s,a}^{\nabla^2V_a}(\nabla f)\right]. \]
	Since $f$ is non-decreasing, all entries of $\nabla f$ are non-negative, and since $a$ satisfies \emph{(BM)}, Lemma \ref{lem_BM} implies that all entries of $\calP_{t-s,a}^{\nabla^2V_a}(\nabla f)$ are non-negative. Moreover, $a$ is positive and non decreasing, so that
	\[ \frac{\nabla a}{a}\cdot\calP_{t-s,a}^{\nabla^2V_a}(\nabla f)\geq0. \]
	Hence, since $\P_s$ preserves the positivity, $\Psi'(s)\geq0$ and the proof is over.
\end{proof}

\begin{req}
	In dimension $d=1$, due to the particular form of the Feynman-Kac semigroup $(\calP_{t,a}^{\nabla^2V_a})_{t\geq0}$, Proposition \ref{prop_monotone} still holds if one only assumes that $a$ is positive and $a$ and $f$ are both non-decreasing.
\end{req}

Proposition \ref{prop_monotone} enables us to adapt the proof of Theorem \ref{thm_LSI} and improve the estimate on $c_{LSI}(\mu)$. Moreover, the proof allows to handle unbounded perturbation functions (as long as the $(G)$ condition is satisfied).

\begin{thm}\label{thm_mono}
	Let $a\in\mathcal C_+^\infty(\R^d,\R)$ be non-decreasing. Define 
	\[\tilde\kappa_a=\inf\limits_{x\in\R^d}\left\{\rho_-(\nabla^2V(x))-a\L(a\Inv)(x)\right\}.\]
	If $a$ satisfies (BM), (G) and $\tilde\kappa_a>0$, then for any non-decreasing $f\in\mathcal C_+^\infty(\R^d,\R)$,
	\[\ent_\mu(f^2)\leq\frac{2}{\tilde\kappa_a}\int_{\R^d}|\nabla f|^2\;d\mu.\]
\end{thm}
\begin{proof}
	Let $f\in\mathcal C^\infty_+(\R^d,\R)$ be non-decreasing. The beginning of the proof is very similar to the one of Theorem \ref{thm_LSI}. Indeed, the entropy rewrites
	\[\ent_\mu(f)=\int_{\R^d}\int_0^{+\infty}\frac{\left|\calP_t^{\nabla^2V}(\nabla f)\right|^2}{\P_tf}\;dt\;d\mu,\]
	with the representation
	\[\calP_t^{\nabla^2V}(\nabla f)=2\esp\left[R_{t,a}J_t^{X_a}\nabla \sqrt f(X_{t,a})\sqrt f(X_{t,a})\right],\]
	since $a$ satisfies $(G)$. Theorem \ref{cor_WFK} and Cauchy-Schwartz' inequality imply here
	\begin{align*}
		\left|\calP_t^{\nabla^2V}(\nabla f)\right|^2&\leq4\esp\left[R_{t,a}^2|J_t^{X_a}\nabla \sqrt f(X_{t,a})|^2\right]\overbrace{\esp[f(X_{t,a})]}^{\P_{t,a}f}\\
		&\leq4\esp\left[R_{t,a}^2|J_t^{X_a}\nabla \sqrt f(X_{t,a})|^2\right]\P_tf,
	\end{align*}
	using Proposition \ref{prop_monotone}. Plugged into the entropy, this yields
	\[\ent_\mu(f)\leq4\int_{\R^d}\int_0^{+\infty}\esp\left[\nabla\sqrt f(X_{t,a})^TJ_t^{X_a}R_{t,a}^2(J_t^{X_a})^T\nabla\sqrt f(X_{t,a})\right]\;dt\;d\mu.\]
	Here we let
	\[J_t^a=J_t^{X_a}\exp\left(-\int_0^t\frac{\L_aa}{a}(X_{s,a})\,ds\right),\]
	and the same reasoning as in the proof of Theorem \ref{thm_LSI} gives then
	\[\ent_\mu(f)\leq4\int_{\R^d}\int_0^{+\infty}e^{-2\tilde\kappa_at}\esp\left[\frac{a(X_{t,a})^2}{a(x)^2}|\nabla\sqrt f(X_{t,a})|^2\right]\;dt\;d\mu.\]
	Hence, using $\mu_a$-invariance of $(\P_{t,a})_{t\geq0}$,
	\begin{align*}
		\ent_\mu(f)&\leq4\int_0^{+\infty}e^{-2\tilde\kappa_at}\int_{\R^d}\P_{t,a}\left(a^2|\nabla\sqrt  f|^2\right)\;d\mu_a\;dt\\
		&=4\int_0^{+\infty}e^{-2\tilde\kappa_at}\int_{\R^d}|\nabla\sqrt  f|^2\;d\mu\;dt=\frac{2}{\tilde\kappa_a}\int_{\R^d}|\nabla\sqrt  f|^2\;d\mu,
	\end{align*}
	and the proof is achieved replacing $f$ by $f^2$.
\end{proof}

\section{Examples}

In this section, we illustrate the Feynman-Kac approach on some examples. Since the perturbation function we introduce is scalar-valued, the method will be particularly suitable for potentials whose Hessian matrix admits many symmetries, for instance radial potentials. The examples we focus on shall then pertain to this class of potentials, namely here Subbotin and double-well potentials. Let us mention that, using other techniques, similar results for compactly supported radial measures were recently derived by Cattiaux, Guillin and Wu in \cite{CGW19}.

For the sake of concision, we restrain ourselves to the illustration of Theorem \ref{thm_LSI}. We eventually briefly resume the comparison to Holley-Stroock method.

\subsection{Subbotin potentials}
The first example we focus on is the general Subbotin\footnote{after Mikhail Fedorovich Subbotin, 1893--1966, Soviet mathematician} distribution  \cite{Sub23}. We take then $V(x)=|x|^\alpha/\alpha$ for $\alpha>2$, to ensure that $V$ is at least twice continuously differentiable, but Bakry-Émery criterion does not apply (see the following proof and remark thereafter).
\begin{lem}\label{lem_eigenval}
Let $a\in\mathcal C_+^\infty(\R^d,\R)$. Then for any $x\in\R^d$,
\[\rho_-(2\nabla^2V(x))-a\L(a\Inv)(x)=2|x|^{\alpha-2}-a\L(a\Inv)(x).\]
\end{lem}
\begin{proof}
First, notice that for any fixed $x\in\R^d$,
\[\nabla^2V(x)=(\alpha-2)|x|^{\alpha-4}xx^T+|x|^{\alpha-2}I_d.\]
Hence, $T_x:=2\nabla^2V(x)-a\L(a\Inv)(x)I_d$ (seen as an element of $\mathcal L(\R^d)$), can be written as the sum of a rank 1 operator (projection on $\R x$) and a full-rank operator (multiple of the identity). One can then write $\R^d=\R x\oplus(\R x)^\bot$. Let $\lambda$ be a non-zero eigenvalue of $T_x$ and $y$ be an associated eigenvector. Then
\begin{itemize}
\item either $y\in\R x$, that is, $y=\beta x$ for some $\beta\in\R^*$, and one can write
\[\lambda y=T_xy=2\beta(\alpha-2)|x|^{\alpha-2}x+2\beta|x|^{\alpha-2}x-\beta a\L(a\Inv)(x)x,\]
which leads to
\[\lambda=2(\alpha-1)|x|^{\alpha-2}-a\L(a\Inv)(x);\]
\item either $y\in(\R x)^\bot$, in which case
\[\lambda y=T_xy=2|x|^{\alpha-2}y-a\L(a\Inv)(x)y,\]
which entails
\[\lambda=2|x|^{\alpha-2}-a\L(a\Inv)(x).\]
\end{itemize}
Hence for any $x\in\R^d$, since $\alpha>2$,
\[\rho_-(2\nabla^2V(x))-a\L(a\Inv)(x)=\rho_-(T_x)=2|x|^{\alpha-2}-a\L(a\Inv)(x).\]
\end{proof}

In the following, we may focus on the $\alpha=4$ (quadric) case. Indeed, computations turn out to be particularly difficult in full generality, as well as keeping track of dependency in both parameters $d$ and $\alpha$. Bakry-Émery criterion clearly does not apply to this particular potential, since $\rho_-(\nabla^2V(x))$ vanishes at point $x=0$.

\begin{thm}\label{thm_subb}
Let $\mu(dx)\propto \exp(-|x|^4/4)dx$. There exists a universal explicit constant $c>0$ such that for any $f\in\mathcal C_c^\infty(\R^d,\R)$, one has
\[\ent_\mu(f^2)\leq c\int_{\R^d}|\nabla f|^2d\mu.\]
In particular, $c$ does not depend on the dimension.
\end{thm}
From the proof, one infers that $c=3e\sqrt 3$ is a suitable constant, yet highly dependent on the way computations are handled.
\begin{proof}
The first concern about making use of Theorem \ref{thm_LSI} stands in the choice of the perturbation function $a$. In practice, $a$ should correct a lack of convexity of $V$ where it occurs (namely where $\nabla^2V(x)\leq0$, here at $x=0$). One of the first choices turns out to be the function $$a(x)=\exp\left(\frac{\varepsilon}{2}\arctan(|x|^2)\right),\quad x\in\R^d.$$
Indeed, the arctangent function behaves like the identity near zero (where lies the lack of convexity of $V$) and like a constant at infinity (ensuring that $a$ is bounded above and below). Furthermore, the square function is uniformly convex on $\R^d$, so that the Hessian matrix of the above is positive definite near the origin. Finally, taking exponential, $a$ is indeed positive and computations are easier. Note that this choice is motivated by some results on the spectral gap, in which case the choice of a perturbation function that is close to non-integrability can provide relevant estimates on the Poincaré constant (see for example \cite{BJ14,ABJ18}).
  
The next step in the method consists in the explicit computation of $\kappa_a$. With this definition of $a$, one has
\[-a\L(a\Inv)(x)=\varepsilon\frac{d+|x|^4(d-4)}{(1+|x|^4)^2}-\varepsilon\frac{|x|^4}{1+|x|^4}-\varepsilon^2\frac{|x|^2}{(1+|x|^4)^2},\quad x\in\R^d,\]
and shall then minimize in $x\in\R^d$:
\[2|x|^2+\varepsilon\frac{d+|x|^4(d-4)}{(1+|x|^4)^2}-\varepsilon\frac{|x|^4}{1+|x|^4}-\varepsilon^2\frac{|x|^2}{(1+|x|^4)^2},\]
which rewrites, setting $t=|x|^2$,
\[\kappa_a=\inf\limits_{t\geq0}\left(2t+\varepsilon\frac{d+t^2(d-4)}{(1+t^2)^2}-\varepsilon\frac{t^2}{1+t^2}-\varepsilon^2\frac{t}{(1+t^2)^2}\right).\]

Optimization of polynomials is hardly explicit in most cases, especially when one must keep track of all parameters (namely $\varepsilon$ and $d$). We shall then focus here on the case where the infimum is reached for $t=0$, that is, for any $t\geq0$,
\[2t^4-\varepsilon(d+1)t^3+4t^2-\varepsilon(d+5)t+2-\varepsilon^2\geq0.\]

Let us denote by $g$ the above polynomial function. Clearly, $\varepsilon\leq\sqrt2$ is a necessary, yet not sufficient condition for $g$ to be non negative. In order to make computations more tractable, let us assume that $g''$ is positive. This is true as soon as
\[\varepsilon<\frac{8}{\sqrt3(d+1)}.\]
Consider then $\varepsilon\leq\dfrac{8}{3\sqrt3(d+1)}$. With this choice of $\varepsilon$, given that $d\geq1$, one has for any $t\geq0$
\[g(t)\geq2t^4-\frac{8t^3}{3\sqrt3}+4t^2-\frac{8t}{\sqrt3}+2-\frac{16}{27}.\]
It is easy to see that the above right-hand side is non-negative, so that $g$ is non-negative either over $\R_+$. We can then take $\kappa_a=\varepsilon d$, and Theorem~\ref{thm_LSI} entails the following estimate:
\[c_{LSI}(\mu)\leq\frac{4e^{\varepsilon\pi/4}}{\varepsilon d},\]
with $\varepsilon\leq\dfrac{8}{3\sqrt3(d+1)}$ (which implies that $\varepsilon\leq\sqrt2$). We finally minimize this bound with respect to $\varepsilon\in\left(0,\frac{8}{3\sqrt3(d+1)}\right]$ to get
\[c_{LSI}(\mu)\leq\frac{3\sqrt3(d+1)}{2d}e^{2\pi/3\sqrt3(d+1)}.\]
The above is uniformly bounded with respect to $d\in\N^\star$, and one can take $c=3e\sqrt3$ as the universal constant mentioned in the theorem.
\end{proof}

\begin{req}
This proof points out the main concerns about Theorem \ref{thm_LSI}. Indeed, the choice of the function (or family of functions) $a$ is a key point. Nevertheless, the most important, yet technical, part of the proof is the explicit computation of $\kappa_a$, given that track should be kept of all parameters. 

Note that, up to some numerical constant, the bound on $\varepsilon$ in the previous proof is optimal (with this optimization method). Recall that the problem reduces to the prove that the function $g$ defined on $\R_+$ as
\[g(t)=2t^4-\varepsilon(d+1)t^3+4t^2-\varepsilon(d+5)t+2-\varepsilon^2,\quad t\geq0,\]
is non-negative. If we assume that $\varepsilon$ is of order $(d+1)^{-r}$ for some $r\in(0,1)$, then when $d$ is large, for any fixed positive $t$,
\[g(t)\sim2t^4-d^{1-r}t^3+4t^2-d^{1-r}t+2-d^{-2r},\]
and taking $t=3/d^{1-r}$ leads to
\[g(3/(d+1)^r)\sim\frac{162}{d^{4(1-r)}}+\frac9{d^{2(1-r)}}-\frac1{2^{2r}}-1<0\]
when $d$ increases, which prevents the infimum of $t\mapsto\varepsilon d+tg(t)$ to be reached at $t=0$.

We do not know if the constant we inferred is optimal (in terms of the dimension). Yet, one can note that, for example from \cite{BJM16}, since the spectral gap for the quadric Subbotin distribution is of order $\sqrt{d}$, it is reasonable to expect $c_{LSI}(\mu)$ to be of order $1/\sqrt d$ (since $\mu$ satisfies a Poincaré inequality with constant $c$ (which is the inverse of the spectral gap) as soon as it satisfies a $LSI$ with constant $2c$, see \cite{BGL14} \S 5.1.2). It is then unclear that we can reach optimality with this very optimization procedure. More reliable optimization techniques would be then a good improvement regarding explicit estimates using this result.
\end{req}

\begin{req}
The Holley-Stroock method as developed in Remark~\ref{req_HS} leads, in the present case and after tedious computations, to a conclusion somewhat comparable to ours. Nevertheless, the involved constants are not fully explicit and leave less room for improvement than our above approach. 
\end{req}

\subsection{Double-well potentials}
The following example is a perturbation of the previous one, known as the double-well potential. Consider $V(x)=|x|^4/4-\beta|x|^2/2$, where $\beta>0$ controls the size of the concave region. Although $V$ is convex at infinity, its Hessian matrix is negative definite near the origin, and Bakry-Émery criterion does not apply. Still, one can expect to recover the behaviour inferred in Theorem \ref{thm_subb} when $\beta$ is small.

Similarly to the Subbotin case, one can explicitly compute the Hessian matrix of $V$ to get the following lemma.
\begin{lem}
Let $a\in\mathcal C_+^\infty(\R^d,\R)$. Then for any $x\in\R^d$,
\[\rho_-(2\nabla^2V(x)-a\L(a\Inv)(x)I_d)=2|x|^2-2\beta-a\L(a\Inv)(x).\]
\end{lem}
\begin{proof}
The proof is identical to the one of Lemma~\ref{lem_eigenval}.
\end{proof}
\begin{thm}
Let $\mu(dx)\propto\exp(-|x|^4/4+\beta|x|^2/2)dx$, $\beta\in(0,1/2)$. There exists $c_\beta>0$ a universal constant such that, for any function $f\in\mathcal C_c^\infty(\R^d,\R)$, one has
\[\ent_\mu(f^2)\leq c_\beta\int_{\R^d}|\nabla f|^2d\mu.\]
Again, $c_\beta$ does not depend on the dimension.
\end{thm}
A suitable constant is here $c_\beta=\frac{4e}{1-2\beta}$, the blow-up when $\beta\to 1/2$ is a computation artefact and has, to our knowledge, no qualitative significance.  
\begin{proof}
This proof is very similar to the previous one. In particular, we set for any $x\in\R^d$
\[a(x)=\exp\left(\frac{\varepsilon}{2}\arctan(|x|^2)\right),\]
so that, for $t=|x|^2$,
\[\kappa_a=\inf\limits_{t\geq0}\left(2t-2\beta+\varepsilon\frac{d+t^2(d-4)}{(1+t^2)^2}-\varepsilon(t-\beta)\frac{t}{1+t^2}-\varepsilon^2\frac{t}{(1+t^2)^2}\right).\]
Again, we aim to show that this infimum is equal to $\varepsilon d-2\beta$, reached for $t=0$, which amounts to prove that, for any $t\geq0$,
\[g(t):=2t^4-\varepsilon(d+1)t^3+(4+\beta)t^2-\varepsilon(d+5)t+2-\varepsilon^2+\beta\geq0,\]
along with, to ensure positivity of $\kappa_a$, $\varepsilon>2\beta/d$.

The first necessary condition that arises is $\varepsilon\leq\sqrt{\beta+2}$. Moreover, in light of both previous proof and remark, $\varepsilon$ should be of order $\frac1{d+1}$. To make computations easier, we take $\varepsilon=\frac2{d+1}$. Plugging this into both conditions $\varepsilon>2\beta/d$ and $\varepsilon\leq\sqrt{\beta+2}$ imply that $\beta$ should not exceed $d/d+1$ for any $d$, which equates to $\beta<1/2$. To summarize, we have
\[\varepsilon=\frac2{d+1}\quad \mathrm{and}\quad 0\leq\beta<\frac12.\]
Under those assumptions, $g$ can be bounded from below as follows
\[g(t)\geq 2t^4-2t^3+4t^2-2t+1+\beta,\quad t\geq0.\]
The right-hand term is positive on $\R_+$, so that with this choice of $\varepsilon$, one has
\[\kappa_a=\frac{2d}{d+1}-2\beta.\]
This amounts, using Theorem \ref{thm_LSI}, 
\[c_{LSI}(\mu)\leq\frac{4(d+1)}{2d(1-\beta)-2\beta}e^{\frac{\pi}{2(d+1)}}.\]
The above is uniformly bounded with respect to $d\in\N^\star$, and one can take $c_\beta=\dfrac{4e}{1-2\beta}$ as the aforementioned universal constant.
\end{proof}
\begin{req}
Note that the restriction on $\beta$ is a computation artefact, and one has more $c_\beta\xrightarrow[\beta\to\frac12^-]{}+\infty$. Nevertheless, the behaviour in term of the dimension is similar to what was derived for the Subbotin distribution in Theorem \ref{thm_subb}. 
\end{req}

\begin{req}
As for the quadric distribution, the Holley-Stroock method provides somewhat similar results, yet computations are far more tedious in this case, particularly in keeping track of the dependency in $\beta$.

\end{req}

\section*{Acknowledgements}
The author is highly grateful to his PhD advisor Aldéric Joulin, for the introduction to the subject and all the interesting discussions, and to the referee for the numerous and very helpful comments. He also acknowledges the partial support of the grant ANR-18-CE40-0006 MESA funded by the French National Research Agency (ANR).

~\

\noindent\textbf{Contact informations}: UMR CNRS 5219, Institut de Mathématiques de Toulouse, Université Toulouse III Paul-Sabatier, Toulouse, France

\noindent E-mail: \verb;clement.steiner@math.univ-toulouse.fr;

\noindent URL: \verb?https://perso.math.univ-toulouse.fr/csteiner/?
\end{document}